\documentclass[11pt]{article}

\usepackage[a4paper,margin=1in]{geometry}
\usepackage[T1]{fontenc}
\usepackage{lmodern}
\usepackage{amsmath,amssymb,amsthm,mathtools}
\usepackage{enumitem}
\usepackage{microtype}
\usepackage[colorlinks=true,linkcolor=blue,citecolor=blue,urlcolor=blue]{hyperref}

\newtheorem{theorem}{Theorem}[section]
\newtheorem{proposition}[theorem]{Proposition}
\newtheorem{lemma}[theorem]{Lemma}
\newtheorem{corollary}[theorem]{Corollary}
\newtheorem{remark}[theorem]{Remark}

\newcommand{\g}{\mathfrak g}
\newcommand{\Rad}{\operatorname{Rad}}
\newcommand{\cO}{\mathcal O}

\title{Odd-rank maximal ideals at collapsing levels of type $D$}
\author{Sihai Jin}
\date{}

\begin{document}
\maketitle

\begin{abstract}
We determine the defining ideal of the simple affine vertex algebra
$L_{2-\ell}(\mathfrak{so}_{2\ell})$ for every odd $\ell\ge5$.
Per\v{s}e's quadratic singular vector alone generates the maximal ideal of
the universal affine vertex algebra at this level.  Together with the
established even-rank presentation, this gives a complete parity-dependent
description of this type-$D$ collapsing family at $k=2-\ell$: one quadratic
generator in
odd rank, and a quadratic generator together with two Pfaffian generators
in even rank.  The proof establishes a rank reduction for the quadratic
quotients under minimal Drinfeld--Sokolov reduction, valid in both parities.
Nonvanishing of reduction on nonzero graded subquotients then lifts
simplicity along the odd-rank chain from the known base case $D_3\cong A_3$
at level $-1$.
\end{abstract}



\section{Introduction}

An explicit presentation of an algebra separates relations that are known
to hold from relations that suffice to define it.  For affine vertex
algebras, this distinction is the problem of determining the maximal ideal
of the universal vacuum module.  Let $V^k(\g)$ be the universal affine
vertex algebra of a complex simple Lie algebra $\g$ at level $k$, and let
$L_k(\g)$ be its simple graded quotient.  A set of generators for
$\Rad V^k(\g):=\ker(V^k(\g)\to L_k(\g))$ gives all defining relations for
$L_k(\g)$ and identifies which $V^k(\g)$-modules factor through it.
At negative non-admissible levels, constructing a singular vector is often
more accessible than proving that no further generators are needed.

We address this problem for $D_\ell=\mathfrak{so}_{2\ell}$ at
$k=2-\ell=1-h^\vee/2$, where $h^\vee=2\ell-2$ is the dual Coxeter number.
This is a collapsing level: the simple minimal $W$-algebra obtained by
Hamiltonian reduction is generated by its conformal-weight-one currents
\cite[Theorem~3.3 and Table~4]{AKMPPStructural}.  Such a collapse simplifies the reduced algebra,
but does not by itself give the defining ideal of the original affine
algebra.  The issue is to carry the proposed relations through reduction
and then recover simplicity before reduction.

Use the roots $\pm\varepsilon_i\pm\varepsilon_j$ of $D_\ell$, where the
$\varepsilon_i$ form an orthonormal basis.  For a root $\alpha$, write
$e_\alpha$ for the root vector normalized in Section~\ref{sec:conventions},
and $e_\alpha(n)=e_\alpha\otimes t^n$ for its current mode.  The symbol
$\mathbf1$ denotes the vacuum.  Per\v{s}e
\cite[Theorem~3.1]{PerseD} constructed the quadratic singular vector
\begin{equation}\label{eq:vell}
 v_\ell=\sum_{i=2}^{\ell}
 e_{\varepsilon_1-\varepsilon_i}(-1)
 e_{\varepsilon_1+\varepsilon_i}(-1)\mathbf1
 \in V^{2-\ell}(D_\ell).
\end{equation}
Here and below $\langle v\rangle$ denotes the vertex-algebra ideal generated
by $v$.  Put $Q_\ell=V^{2-\ell}(D_\ell)/\langle v_\ell\rangle$.
This ideal is proper, as recalled in Lemma~\ref{lem:proper}, so there is a
canonical surjection $Q_\ell\twoheadrightarrow L_{2-\ell}(D_\ell)$.
Our main result proves that the quadratic relation is complete in odd rank.

\begin{theorem}\label{thm:main}
Let $\ell\ge5$ be odd.  Then
$$
 \operatorname{Rad}V^{2-\ell}(D_\ell)=\langle v_\ell\rangle.
$$
Equivalently, $Q_\ell\cong L_{2-\ell}(D_\ell)$.
\end{theorem}

For comparison, the even-rank presentation of
\cite[\S8.2, Theorem~8.6]{AKMPP} includes two Pfaffian singular vectors.
For even $\ell$, let $B_\ell$ be the skew-symmetric $\ell\times\ell$ matrix
with $(B_\ell)_{ij}=e_{\varepsilon_i+\varepsilon_j}(-1)$ for $i<j$, and set
$w_\ell^+=\operatorname{Pf}(B_\ell)\mathbf1$ and
$w_\ell^-=\tau(w_\ell^+)$.
Here $\operatorname{Pf}$ is the usual Pfaffian and $\tau$ is the diagram
automorphism induced by $\varepsilon_\ell\mapsto-\varepsilon_\ell$, fixing
the other coordinates.  The entries of $B_\ell$ commute, so its Pfaffian
has its ordinary polynomial meaning.  These vectors have affine conformal
degree $\ell/2$ and finite highest weights
$\varepsilon_1+\cdots+\varepsilon_{\ell-1}\pm\varepsilon_\ell$,
respectively.  The two parities can now be stated together.

\begin{corollary}[Parity dichotomy]\label{cor:parity-intro}
For every $\ell\ge4$,
$$
 \operatorname{Rad}V^{2-\ell}(D_\ell)=
 \begin{cases}
   \langle v_\ell\rangle, & \ell\ \text{odd},\\[1mm]
   \langle v_\ell,w_\ell^+,w_\ell^-\rangle, & \ell\ \text{even}.
 \end{cases}
$$
The even line is \cite[Theorem~8.6]{AKMPP}; the odd line is
Theorem~\ref{thm:main}.
\end{corollary}

The quotients $Q_\ell$ already occur in several parts of the theory.
Per\v{s}e introduced them and studied their representations in
\cite{PerseD}.  Adamovi\'c--Per\v{s}e treated the odd-rank simple algebras
and their fusion rules in \cite{AdamovicPerseFusion}; Arakawa--Moreau
determined the associated varieties of the corresponding simple affine vertex
algebras in \cite[\S9]{ArakawaMoreau}.
The irreducible $Q_\ell$-modules in the category with locally finite
$D_\ell$-action are recalled in \cite[Proposition~8.1]{AKMPP}.
These representation-theoretic and geometric descriptions do not by
themselves exclude a further proper ideal in $Q_\ell$.
Thus the issue is not the existence of the familiar quadratic relation,
but whether that single relation is already complete.
Theorem~\ref{thm:main} supplies the missing algebra presentation: in odd
rank the familiar quadratic quotient is already simple.

The proof uses minimal Drinfeld--Sokolov reduction, denoted by
$H^0_{\mathrm{DS},f_\theta}$, for a minimal nilpotent element $f_\theta$.
The centralizer $\g^\natural$ of its $\mathfrak{sl}_2$-triple is
$A_1\oplus D_{\ell-2}$; the corresponding currents in the reduced algebra
have levels $0$ and $4-\ell$.  We show that the cohomology class of
$v_\ell$ is a nonzero multiple of the positive-root current in the
$A_1$ factor.  Setting this class to zero eliminates that entire factor
and the conformal-weight-$3/2$ generators.  Their product relations then
force precisely the quadratic relation for $D_{\ell-2}$.  This gives a
surjective vertex-algebra homomorphism
\begin{equation}\label{eq:rank-map-intro}
 Q_{\ell-2}\twoheadrightarrow H^0_{\mathrm{DS},f_\theta}(Q_\ell).
\end{equation}
The calculation holds for every $\ell\ge5$, independently of parity.
At our negative levels, minimal reduction detects every nonzero graded
subquotient of the vacuum module.  Exactness therefore lifts simplicity of
the target in \eqref{eq:rank-map-intro} back to $Q_\ell$.
For odd rank, successive minimal reductions reach $Q_3$ at level $-1$, whose simplicity is
the type-$A_3$ case of \cite[Theorem~7.2]{ArakawaMoreau}.

This method is related to the even-rank reduction in \cite[\S8]{JinAM},
through the local quadratic-reduction framework.  The odd-rank
maximal-ideal theorem proved here is logically independent of the
maximal-ideal theorem of \cite{JinAM}: none of the latter theorem's
conclusions is used below.  The contribution here is to isolate the local
quadratic calculation in a form valid in both parities, include the endpoint
$D_5\to D_3$, and combine it with a self-contained lifting argument and the
published $A_3$ base case to prove the odd-rank theorem.

Section~\ref{sec:conventions} fixes notation and normalizations.
Section~\ref{sec:maximality} gives the lifting criterion,
Section~\ref{sec:rank} proves the rank reduction, and
Section~\ref{sec:odd} completes the induction.

\section{Conventions and quadratic quotients}
\label{sec:conventions}

All Lie algebras and vertex algebras are over $\mathbb C$.
For a simple Lie algebra $\g$, fix a Cartan subalgebra $\mathfrak h$,
a set of positive roots, and its highest root $\theta$.
The invariant bilinear form $(\cdot|\cdot)$ is normalized so that long
roots have squared length $2$.  The root space of $\alpha$ is denoted by
$\g_\alpha$.  We write $P_+$ for the dominant integral weights and
$\varpi_i$ for the fundamental weights of the specified root system;
$U(\mathfrak a)$ denotes the universal enveloping algebra of a Lie algebra
$\mathfrak a$.

The untwisted affine Lie algebra, including its degree derivation $d$, is
$$
 \widehat\g=(\g\otimes\mathbb C[t,t^{-1}])\oplus\mathbb CK\oplus\mathbb Cd,
 \qquad [d,a(n)]=n a(n),
$$
where $a(n)=a\otimes t^n$ and $K$ is central.  Its other brackets are
$[a(m),b(n)]=[a,b](m+n)+m\delta_{m+n,0}(a|b)K$, with
$\delta_{m+n,0}$ the Kronecker delta.
The affine Cartan is $\widehat{\mathfrak h}=\mathfrak h\oplus\mathbb CK
\oplus\mathbb Cd$.  Extend finite weights by zero on $K,d$, and normalize
$\Lambda_0$ and the null root $\delta$ by
$\Lambda_0(K)=1$, $\Lambda_0(\mathfrak h\oplus\mathbb Cd)=0$,
$\delta(d)=1$, and $\delta(\mathfrak h\oplus\mathbb CK)=0$.
The additional simple root and coroot are
$\alpha_0=\delta-\theta$ and $\alpha_0^\vee=K-\theta^\vee$.
Write $\widehat Q_+$ for the monoid generated by the affine simple roots.

The vacuum in $V^k(\g)$ has degree zero and $d$ acts as minus the degree.
A mode $a(n)$ lowers degree by $n$.  All homogeneous components of the
vacuum module are finite-dimensional $\g$-modules.  At noncritical level
$k\ne-h^\vee$, this degree agrees with the Sugawara conformal degree.
Every subquotient considered below carries the inherited grading.  We call a
graded module $M$ \emph{lower-bounded} if there exists $s\in\mathbb C$ such
that
$$
 M=\bigoplus_{n\ge0} M[s+n].
$$
A vector is \emph{affine singular} if it is annihilated by
$\g\otimes t\mathbb C[t]$ and by the positive-root zero modes.

For a vertex-algebra state $a$, our conventions are
$$
 Y(a,z)=\sum_{n\in\mathbb Z}a_{(n)}z^{-n-1},\qquad
 [a{}_{\lambda}b]=\sum_{n\ge0}\frac{\lambda^n}{n!}a_{(n)}b,
 \qquad :ab:=a_{(-1)}b.
$$
The symbol $\lambda$ is a formal variable and $\partial$ is the translation
operator.  For a current state $a(-1)\mathbf1$, its vertex mode $a_{(n)}$
is the affine mode $a(n)$.  We use the same symbol for a state and its
field when no confusion can arise.  Strong generation means that normally
ordered products of derivatives of the indicated states span the algebra.

For $D_\ell$, including $D_3\cong A_3$, use positive roots
$\varepsilon_i\pm\varepsilon_j$ with $i<j$.  The simple roots are
$\varepsilon_i-\varepsilon_{i+1}$ for $1\le i<\ell$ and
$\varepsilon_{\ell-1}+\varepsilon_\ell$.
Index matrix rows and columns by $1,\ldots,\ell,-1,\ldots,-\ell$, and let
$E_{a,b}$ be the matrix unit with entry $1$ in position $(a,b)$.
We use
\begin{equation}\label{eq:root-normalization}
 \begin{aligned}
 e_{\varepsilon_i-\varepsilon_j}&=E_{i,j}-E_{-j,-i},&
 e_{\varepsilon_i+\varepsilon_j}&=E_{i,-j}-E_{j,-i}\quad(i<j),\\
 e_{-\alpha}&=-e_\alpha^{\mathsf T}\quad(\alpha>0),&
 (a|b)&=\tfrac12\operatorname{tr}(ab).
 \end{aligned}
\end{equation}
Here $\mathsf T$ denotes matrix transpose.  If $H_i=E_{i,i}-E_{-i,-i}$,
then $\varepsilon_i(H_j)=\delta_{ij}$.
For $\theta=\varepsilon_1+\varepsilon_2$ put
$f_\theta=-e_{-\theta}$ and $h_\theta=H_1+H_2$.
Thus $[e_\theta,f_\theta]=h_\theta$ and $(e_\theta|f_\theta)=1$.
The automorphism $\tau$ in the introduction is induced by conjugation with
the permutation matrix interchanging the coordinate vectors indexed by
$\ell$ and $-\ell$.

Formula \eqref{eq:vell} defines $v_\ell$ and $Q_\ell$ also for $\ell=3$
(verified in Section~\ref{sec:base}).  For $\ell\ge4$,
\cite[Theorem~3.1]{PerseD} checks the affine Chevalley generators; with
\eqref{eq:root-normalization} the finite simple zero modes cancel and
$f_\theta(1)v_\ell=(k_\ell+\ell-2)e_{\varepsilon_1-\varepsilon_2}(-1)\mathbf1=0$
for $k_\ell=2-\ell$, so the same proof applies.

\begin{lemma}\label{lem:proper}
For $\ell\ge5$, the ideal $\langle v_\ell\rangle$ is a proper graded ideal and equals
the affine submodule $U(\widehat{D}_\ell)v_\ell$.
\end{lemma}

\begin{proof}
Let $M_\ell=U(\widehat{D}_\ell)v_\ell$.  By construction $M_\ell$ is graded
and stable under every current mode.  Since $M_\ell$ is a restricted affine
module and $k+h^\vee=(2-\ell)+(2\ell-2)=\ell\ne0$, the Sugawara formula
for $L_{-1}=T$ is locally finite on each vector and is a sum of compositions
of current modes.  Hence $T M_\ell\subset M_\ell$.

The currents strongly generate $V^{2-\ell}(D_\ell)$; the derivative and
normal-product mode formulas therefore show inductively that current-mode and
$T$-stability imply stability under every state mode.  Thus $M_\ell$ is a
vertex ideal.  Conversely every vertex ideal containing $v_\ell$ is
current-mode stable, so
$$
 M_\ell=\langle v_\ell\rangle.
$$

It remains to see that this ideal is proper.  Each summand of
\eqref{eq:vell} has finite weight $2\varepsilon_1$, and $v_\ell$ has degree
$2$; together with its affine singularity this makes it a highest-weight
vector of affine weight $(2-\ell)\Lambda_0+2\varepsilon_1-2\delta$.
The Poincar\'e--Birkhoff--Witt theorem therefore gives
$$
 M_\ell
 =U\bigl((D_\ell\otimes t^{-1}\mathbb C[t^{-1}])
       \oplus\mathfrak n_-\bigr)v_\ell,
$$
where $\mathfrak n_-\subset D_\ell$ acts by zero modes.  Negative loop
modes strictly increase degree, while these zero modes preserve it.  Hence
every vector of $M_\ell$ has degree at least $2$, so $\mathbf1\notin M_\ell$
and $\langle v_\ell\rangle$ is proper.
\end{proof}

\section{A minimal-reduction lifting criterion}
\label{sec:maximality}

Let $\g$ be simple and choose a minimal $\mathfrak{sl}_2$-triple
$(e_\theta,h_\theta,f_\theta)$ with $h_\theta=\theta^\vee$.
Put $x_\theta=h_\theta/2$ and
$\g_j=\{a\in\g:[x_\theta,a]=ja\}$.
The minimal grading and the centralizer of the triple are
$$
 \g=\g_{-1}\oplus\g_{-1/2}\oplus\g_0\oplus\g_{1/2}\oplus\g_1,
 \qquad
 \g^\natural=\{a\in\g:[a,e_\theta]=[a,h_\theta]=[a,f_\theta]=0\}.
$$
In particular $\g_0=\g^\natural\oplus\mathbb Cx_\theta$ is an orthogonal
decomposition.  Write $a^\natural$ for the projection of $a\in\g_0$ onto
$\g^\natural$, and set $\mathfrak h^\natural=\mathfrak h\cap\g^\natural$.

For a level-$k$ affine module $M$, minimal Drinfeld--Sokolov reduction is
the degree-zero cohomology of the standard BRST complex
$$
 C_{\mathrm{DS}}(M)=M\otimes F_{\mathrm{ch}}\otimes F_{\mathrm{ne}},
 \qquad
 H_{\mathrm{DS}}(M):=H^0(C_{\mathrm{DS}}(M),d_{\mathrm{BRST}})
 =H^0_{\mathrm{DS},f_\theta}(M).
$$
Here $F_{\mathrm{ch}}$ and $F_{\mathrm{ne}}$ are the charged and neutral
ghost vertex algebras for $\g_{>0}=\g_{1/2}\oplus\g_1$ and $\g_{1/2}$,
respectively, and $d_{\mathrm{BRST}}$ is the square-zero differential
of \cite[\S1]{KacWakimoto2004}.  Their tensor-product vacuum is
$\mathbf1_{\mathrm{gh}}$.  A vector annihilated by $d_{\mathrm{BRST}}$ is
called BRST closed; its cohomology class is taken modulo the image of
$d_{\mathrm{BRST}}$.  The cohomological degree here is distinct from
conformal degree.  The universal minimal $W$-algebra is
$W^k(\g,f_\theta)=H_{\mathrm{DS}}(V^k(\g))$.

Throughout this section assume $k\in\mathbb Z_{<0}$ and $k\ne-h^\vee$.
Following \cite[\S2.12]{Arakawa2005}, let $\cO_k$ be the full category of
level-$k$ $\widehat\g$-modules $M=\bigoplus_{\lambda(K)=k}M_\lambda$ with
$\dim M_\lambda<\infty$ and $\operatorname{Supp}M\subset
\bigcup_{i=1}^r(\mu_i-\widehat Q_+)$ for some finite set $\{\mu_i\}$.
For the minimal reduction $f=f_\theta$, Theorem~6.7.1 and
Corollary~6.7.3 of \cite{Arakawa2005} give exactness of $H_{\mathrm{DS}}$
on $\cO_k$, while Theorem~6.7.4 gives, for irreducible
$L(\widehat\lambda)\in\cO_k$,
$$
 H_{\mathrm{DS}}(L(\widehat\lambda))\ne0
 \quad\Longleftrightarrow\quad
 \widehat\lambda(\alpha_0^\vee)\notin\mathbb Z_{\ge0}.
$$
The lifting argument below is a consequence of these results; compare
\cite[\S7]{ArakawaMoreau} and \cite[Theorem~2.11]{JinAM}.

\begin{lemma}[Exactness category]\label{lem:exactness-category}
Let $M$ be a lower-bounded graded subquotient of $V^k(\g)$ such that
$$
 M=\bigoplus_{n\ge0}M[s+n],
 \qquad \dim M[s+n]<\infty,
$$
and every homogeneous component is a finite-dimensional $\g$-module.  Then
$M\in\cO_k$.  In particular, minimal Drinfeld--Sokolov reduction is exact on
short exact sequences of such modules.
\end{lemma}

\begin{proof}
The affine Cartan acts semisimply: $K$ acts as $k$, the grading diagonalizes
$d$, and each finite-dimensional homogeneous component is semisimple over the
finite Cartan.  All affine weight spaces are therefore finite-dimensional.
Moreover, $V^k(\g)$ is spanned by PBW monomials
$$
 x_1(-m_1)\cdots x_s(-m_s)\mathbf1,
 \qquad m_i>0,
$$
where each $x_i$ is either a root vector in $\g_{\alpha_i}$ or a Cartan
vector.  In the root-vector case the affine weight contribution is
$\alpha_i-m_i\delta$, so the monomial has affine weight
$$
 k\Lambda_0-\sum_i(m_i\delta-\alpha_i),
$$
with the Cartan contribution interpreted as $-m_i\delta$.  For any finite
root $\alpha$ and $m>0$, $m\delta-\alpha$ is a positive real affine root:
if $\alpha\in\Delta_+$ it is $-\alpha+m\delta$, while if
$\alpha\in\Delta_-$ it is $(-\alpha)+m\delta$.  For a Cartan vector the
contribution is the positive imaginary root $m_i\delta$.  Hence
$$
 \operatorname{Supp}_{\widehat{\mathfrak h}}V^k(\g)
 \subset k\Lambda_0-\widehat Q_+.
$$
Submodules and quotients cannot create new weights, so the same inclusion
holds for $M$.  This is the required support condition for $\cO_k$.
Exactness now follows from \cite[Corollary~6.7.3]{Arakawa2005}.
\end{proof}

\begin{lemma}\label{lem:detection}
Let $k\in\mathbb Z_{<0}$ with $k\ne-h^\vee$, and let $M\neq0$ be a
lower-bounded graded subquotient of $V^k(\g)$ such that each homogeneous
component is a finite-dimensional $\g$-module.  Then
$H_{\mathrm{DS}}(M)\neq0$.
\end{lemma}

\begin{proof}
Choose $s_0$ with $M=\bigoplus_{n\ge0}M[s_0+n]$.  Since $M\ne0$,
there is a least $n_0\ge0$ such that $M[s_0+n_0]\ne0$; put
$s=s_0+n_0$.  Then $M[s-n]=0$ for every integer $n>0$.  Since $\g$ is
semisimple, the finite-dimensional $\g$-module $M[s]$ is completely
reducible.  Choose a highest-weight vector $v$ in an irreducible summand,
with highest weight $\mu\in P_+$.  For $a\in\g$ and $n>0$, one has
$a(n)v\in M[s-n]=0$; positive-root zero modes kill $v$ by construction.
Thus $v$ is an affine highest-weight vector of weight
$\widehat\lambda=k\Lambda_0+\mu-s\delta$ (indeed $dv=-sv$).
Its cyclic module $C=U(\widehat\g)v$ is a quotient of the Verma module
$M(\widehat\lambda)$ and hence has the irreducible quotient
$L(\widehat\lambda)$.  Since $C\subset M$ is $d$-stable, it is graded with
$C[r]\subset M[r]$, hence lower-bounded with finite-dimensional homogeneous
$\g$-components.  The quotient map is $\widehat\g$-linear, so it preserves
both the level $k$ and the $d$-grading; its kernel and quotient have the same
properties.  Thus Lemma~\ref{lem:exactness-category} gives
$C,L(\widehat\lambda)\in\cO_k$.
Since $\delta(\alpha_0^\vee)=0$, while $\mu$ is dominant integral and
$\theta^\vee$ is a positive coroot, one has
$\widehat\lambda(\alpha_0^\vee)=k-\mu(\theta^\vee)<0$.
Arakawa's nonvanishing criterion therefore gives
$H_{\mathrm{DS}}(L(\widehat\lambda))\ne0$.
If $K=\ker(C\twoheadrightarrow L(\widehat\lambda))$, then $K[r]\subset C[r]$
and $(M/C)[r]=M[r]/C[r]$.  Hence $K$ and $M/C$ are again lower-bounded graded
subquotients with finite-dimensional homogeneous $\g$-components, so
Lemma~\ref{lem:exactness-category} applies to
$0\to K\to C\to L(\widehat\lambda)\to0$ and
$0\to C\to M\to M/C\to0$.  Exactness therefore gives a surjection
$H_{\mathrm{DS}}(C)\twoheadrightarrow H_{\mathrm{DS}}(L(\widehat\lambda))$
and an injection $H_{\mathrm{DS}}(C)\hookrightarrow H_{\mathrm{DS}}(M)$.
The first makes $H_{\mathrm{DS}}(C)$ nonzero, and the second proves the claim.
No finite-length assumption on $M$ is used.
\end{proof}

\begin{proposition}[Minimal-reduction maximality principle]
\label{prop:maximality}
Let $k\in\mathbb Z_{<0}$ with $k\ne-h^\vee$.  Let
$N\subset\Rad V^k(\g)$ be a graded ideal and put $Q=V^k(\g)/N$.
If $H_{\mathrm{DS}}(Q)$ is nonzero and simple, then
$Q\cong L_k(\g)$ and $N=\Rad V^k(\g)$.
\end{proposition}

\begin{proof}
Because $N\subset\Rad V^k(\g)=\ker(V^k(\g)\to L_k(\g))$, the quotient map
factors to $q:Q\twoheadrightarrow L_k(\g)$, with
$I:=\ker q=\Rad V^k(\g)/N$.  Since $Q$ is a graded quotient of
$V^k(\g)$, it is lower-bounded of level $k$ with finite-dimensional
homogeneous $\g$-components; the same holds for the graded subquotients
$I$ and $L_k(\g)$.  Lemma~\ref{lem:exactness-category} therefore applies,
and reduction gives
$$
 0\longrightarrow H_{\mathrm{DS}}(I)
 \longrightarrow H_{\mathrm{DS}}(Q)
 \longrightarrow H_{\mathrm{DS}}(L_k(\g))\longrightarrow0.
$$
The last arrow is the surjective vertex-algebra homomorphism induced by
$q\otimes\mathrm{id}$.  Since $L_k(\g)$ is a nonzero graded quotient of
$V^k(\g)$, Lemma~\ref{lem:detection} makes its target nonzero; hence its
kernel is a proper ideal of the simple algebra $H_{\mathrm{DS}}(Q)$ and is
zero.  Exactness gives $H_{\mathrm{DS}}(I)=0$.  If $I\ne0$, then, as the
graded subquotient above, Lemma~\ref{lem:detection} would give
$H_{\mathrm{DS}}(I)\ne0$, a contradiction.
  Hence $I=0$, so
$N=\Rad V^k(\g)$ and $Q\cong L_k(\g)$.
\end{proof}

\section{A parity-free local reduction in type \texorpdfstring{$D$}{D}}
\label{sec:rank}

Fix $\ell\ge5$, put $\g_\ell=D_\ell$ and $k_\ell=2-\ell$, and retain the
root vectors, $\theta=\varepsilon_1+\varepsilon_2$, and $f_\theta$ from
Section~\ref{sec:conventions}.  Set
$W_\ell=W^{k_\ell}(D_\ell,f_\theta)$ and
$R_\ell=H_{\mathrm{DS}}(Q_\ell)$.
Lemmas~\ref{lem:proper} and~\ref{lem:detection} give $R_\ell\ne0$.

The centralizer is $\g_\ell^\natural\cong A_1\oplus D_{\ell-2}$.
The $A_1$ root is $\theta_A=\varepsilon_1-\varepsilon_2$, and the
$D_{\ell-2}$ factor uses the coordinates $\varepsilon_3,\ldots,
\varepsilon_\ell$.
As a module for this direct sum, $(\g_\ell)_{-1/2}$ is the tensor product
of the two-dimensional standard $A_1$-module and the
$(2\ell-4)$-dimensional vector representation of $D_{\ell-2}$.

Since $k_\ell+h^\vee(D_\ell)=\ell\ne0$, the noncritical universal
$W_\ell$ is freely generated by currents $J^{\{a\}}$ for
$a\in\g_\ell^\natural$, generators $G^{\{u\}}$ for
$u\in(\g_\ell)_{-1/2}$, and $\omega$, of conformal weights $1$, $3/2$,
and $2$ \cite[p.~2 and Theorem~2.1]{AKMPPStructural}.  In particular,
$W_\ell[1]=J^{\{\g_\ell^\natural\}}$.  Write
$J_A=J^{\{e_{\theta_A}\}}$.
For a simple factor $\mathfrak s\subset\g_\ell^\natural$,
\cite[Theorem~2.1]{AKMPPStructural} gives level
$k_\ell+(h^\vee(D_\ell)-h^\vee_{0,\mathfrak s})/2$, with
$h^\vee_{0,\mathfrak s}$ computed for the restricted form.  Since
$(H_i|H_j)=\delta_{ij}$, both factors retain the standard normalization
(root length squared $2$), so $h^\vee_{0,A_1}=2$ and
$h^\vee_{0,D_{\ell-2}}=2\ell-6$.  Hence
$k^\natural_{A_1}=0$ and
$k^\natural_{D_{\ell-2}}=4-\ell=k_{\ell-2}$.
In \cite[\S2]{AKMPPStructural}, the root vectors satisfy
$[e_\theta^{\rm AKMPP},e_{-\theta}^{\rm AKMPP}]=x_\theta$.  Since our
normalization gives $[e_\theta,f_\theta]=h_\theta=2x_\theta$, their pair is
$e_\theta^{\rm AKMPP}=\tfrac12e_\theta$, $e_{-\theta}^{\rm AKMPP}=f_\theta$.
Thus below we write $e:=\tfrac12e_\theta$ wherever their $e_\theta$ occurs;
the neutral form is $(f_\theta|[\cdot,\cdot])$.

\begin{proposition}[Local reduction]\label{prop:local-reduction}
For every $\ell\ge5$ the following statements hold.
\begin{enumerate}[label=\textup{(\roman*)},leftmargin=2.2em]
\item In $W_\ell$, the class
$[v_\ell]_{\mathrm{DS}}:=[v_\ell\otimes\mathbf1_{\mathrm{gh}}]$ satisfies
$[v_\ell]_{\mathrm{DS}}=c_\ell J_A$ for some
$c_\ell\in\mathbb C^\times$.  Consequently $J_A=0$ in $R_\ell$.
\item The algebra $R_\ell$ is strongly generated by the surviving
$D_{\ell-2}$ currents.  In particular there is a surjective homomorphism
\begin{equation}\label{eq:universal-surj}
 \psi_\ell:V^{4-\ell}(D_{\ell-2})\twoheadrightarrow R_\ell.
\end{equation}
\item The next quadratic relation lies in the kernel:
$\psi_\ell(v_{\ell-2})=0$.
\end{enumerate}
\end{proposition}

\begin{proof}
\emph{Step 1: the singular class.}
We first verify explicitly that
$v_\ell\otimes\mathbf1_{\mathrm{gh}}$ is BRST closed.  By affine
singularity,
$$
 (\g\otimes t\mathbb C[t])v_\ell=0,
 \qquad \mathfrak n_+(0)v_\ell=0,
$$
where $\mathfrak n_+$ is the positive nilpotent subalgebra for the fixed
choice of positive roots.  Since $\g_{>0}=\g_{1/2}\oplus\g_1\subset
\mathfrak n_+$, also $\g_{>0}(0)v_\ell=0$.  In the explicit BRST field
$d(z)$ of \cite[\S1]{KacWakimoto2004}, the only summand containing an
affine field is $\sum_\alpha u_\alpha(z)\varphi^\alpha(z)$ with
$u_\alpha\in\g_{>0}$.  In its residue, a ghost mode that does not kill
$\mathbf1_{\mathrm{gh}}$ is paired with an affine mode $u_\alpha(m)$,
$m\ge0$; the case $m=0$ is killed by $\g_{>0}(0)v_\ell=0$, and $m>0$ by
affine singularity.  The other three summands of $d(z)$ are ghost-only
fields, whose zero modes annihilate the ghost vacuum by the vacuum axiom.
Hence
$$
 d_{\mathrm{BRST}}(v_\ell\otimes\mathbf1_{\mathrm{gh}})=0.
$$
This is also the closedness used in
\cite[Lemma~7.3(a)]{KacWakimoto2004}.
Write $\xi_\ell=[v_\ell]_{\mathrm{DS}}$.  By
\cite[Lemma~7.3(b), equation~(7.3)]{KacWakimoto2004}, on the ghost vacuum
reduced conformal degree is $\Delta_{\mathrm{aff}}(w)-\mu(x_\theta)$ for
finite weight $\mu$; the ghost vacuum has weight zero.  Here
$\Delta_{\mathrm{aff}}(v_\ell)=2$ and, since $x_\theta=(H_1+H_2)/2$,
$(2\varepsilon_1)(x_\theta)=1$, so the cocycle has reduced degree $1$ and
$\xi_\ell\in W_\ell[1]$.  On $\mathfrak h^\natural=\mathbb C(H_1-H_2)\oplus
\bigoplus_{j=3}^{\ell}\mathbb C H_j$, its finite weight restricts to
$\theta_A$.  Hence the $\theta_A$-line of $W_\ell[1]$ is
$\mathbb C J_A$, so $\xi_\ell\in\mathbb C J_A$.

To prove nonvanishing, we make the classical comparison explicit.
For a vertex algebra $V$ set
$C_2(V)=\operatorname{span}_{\mathbb C}\{a_{(-2)}b:a,b\in V\}$.
The image of $v_\ell$ in
$V^{k_\ell}(\g_\ell)/C_2(V^{k_\ell}(\g_\ell))\cong S(\g_\ell)$ is
the quadratic polynomial
$p_\ell=\sum_{i=2}^{\ell}e_{\varepsilon_1-\varepsilon_i}
e_{\varepsilon_1+\varepsilon_i}$, where $S(\g_\ell)$ is the symmetric
algebra.  Identify $S(\g_\ell)$ with polynomial functions on $\g_\ell$
using $(\cdot|\cdot)$, so that a vector $a$ represents the function
$r_a(s)=(a|s)$.
The Slodowy slice is the affine space
$\mathcal S_{f_\theta}=f_\theta+\g_\ell^{e_\theta}$, with
$\g_\ell^{e_\theta}=\ker(\operatorname{ad}e_\theta)$.

Equip the BRST complex with the Li filtration.  By
\cite[Theorem~4.5.9]{Arakawa2015}, this filtration induces the Li
filtration on the actual BRST cohomology and yields
$\operatorname{gr} W_\ell\cong\mathbb C[\mathcal S_{f_\theta,\infty}]$;
hence \cite[Corollary~4.5.10]{Arakawa2015} gives
$$
 W_\ell/C_2(W_\ell)\cong \mathbb C[\mathcal S_{f_\theta}].
$$
Under this identification, the image of the class of a BRST-closed affine
singular vector is the classical Hamiltonian reduction of its Li symbol,
i.e. its restriction to the Slodowy slice; compare
\cite[Lemma~7.3]{ArakawaMoreau}.  Thus the image of $\xi_\ell$ modulo
$C_2(W_\ell)$ is $p_\ell|_{\mathcal S_{f_\theta}}$.  Consequently, a
nonzero restriction cannot come from a BRST boundary and implies
$\xi_\ell\ne0$.
The minimal $\mathfrak{sl}_2$-decomposition gives
$$
 \g_\ell^{e_\theta}
 =\g_\ell^\natural\oplus(\g_\ell)_{1/2}\oplus\mathbb Ce_\theta.
$$
Invariance gives $(\g_i|\g_j)=0$ unless $i+j=0$.  Thus, for
$s=f_\theta+z$ with $z\in\g_\ell^{e_\theta}$,
$r_{e_\theta}(s)=(e_\theta|f_\theta)=1$.  Moreover every
$a\in(\g_\ell)_{1/2}$ pairs trivially with $f_\theta\in(\g_\ell)_{-1}$
and with $z\in\g_\ell^\natural\oplus(\g_\ell)_{1/2}\oplus
\mathbb Ce_\theta$, so $r_a(s)=0$.  The $i=2$ summand of $p_\ell$ is
$e_{\theta_A}e_\theta$, whereas for $i\ge3$ both factors lie in
$(\g_\ell)_{1/2}$.  Hence
$$
 p_\ell|_{\mathcal S_{f_\theta}}=r_{e_{\theta_A}}|_{\mathcal S_{f_\theta}}.
$$
This function is nonzero: since $e_{-\theta_A}\in\g_\ell^\natural
\subset\g_\ell^{e_\theta}$, the point $f_\theta+e_{-\theta_A}$ lies
on the slice, and by \eqref{eq:root-normalization} its value is
$(e_{\theta_A}|f_\theta)+(e_{\theta_A}|e_{-\theta_A})=0-1=-1$.
Thus $\xi_\ell=c_\ell J_A$ with $c_\ell\ne0$.  Exactness applied to
$0\to\langle v_\ell\rangle\to V^{k_\ell}(D_\ell)\to Q_\ell\to0$ gives
$\pi_\ell:W_\ell\twoheadrightarrow R_\ell$, and naturality gives
$\pi_\ell(\xi_\ell)=0$.  Hence $J_A=0$ in $R_\ell$, proving \textup{(i)}.

\smallskip
\noindent\emph{Step 2: the surviving generators.}
Its kernel is current-mode stable and contains $J_A=J^{\{e_{\theta_A}\}}$.  For $a,b\in A_1\subset\g_\ell^\natural$, the current bracket is
$$
 [J^{\{a\}}{}_{\lambda}J^{\{b\}}]
 =J^{\{[a,b]\}}+k^\natural_{A_1}(a|b)\lambda,
$$
so its zero-mode part is the adjoint action of $A_1$.  Since
$[e_{-\theta_A},e_{\theta_A}]=-h_{\theta_A}$ and
$[e_{-\theta_A},h_{\theta_A}]=2e_{-\theta_A}$, we obtain in $R_\ell$
$$
 0=(J^{\{e_{-\theta_A}\}})_{(0)}J^{\{e_{\theta_A}\}}
   =-J^{\{h_{\theta_A}\}},
 \qquad
 0=(J^{\{e_{-\theta_A}\}})_{(0)}J^{\{h_{\theta_A}\}}
   =2J^{\{e_{-\theta_A}\}}.
$$
Thus all three $A_1$ currents
$J^{\{e_{\theta_A}\}},J^{\{h_{\theta_A}\}},J^{\{e_{-\theta_A}\}}$
vanish in $R_\ell$.

The identity
$[J^{\{a\}}{}_{\lambda}G^{\{u\}}]=G^{\{[a,u]\}}$
holds in the universal minimal $W$-algebra
\cite[Theorem~2.1(c), equation~(2.6)]{AKMPPStructural}.
As an $A_1$-module, $(\g_\ell)_{-1/2}\cong
L_{A_1}(\varpi)^{\oplus(2\ell-4)}$, so it has no trivial summand and
$[A_1,(\g_\ell)_{-1/2}]=(\g_\ell)_{-1/2}$.  Thus for any
$u$ there are $a_i\in A_1$ and $u_i\in(\g_\ell)_{-1/2}$ with
$u=\sum_i[a_i,u_i]$.  Linearity and the zero-mode identity give
$\pi_\ell(G^{\{u\}})=\sum_i\pi_\ell(J^{\{a_i\}})_{(0)}
\pi_\ell(G^{\{u_i\}})=0$,
since every $A_1$ current vanishes in $R_\ell$.  Hence all $G$-fields vanish.

The minimal five-step $\mathfrak{sl}_2$-grading makes
$\operatorname{ad}e=\tfrac12\operatorname{ad}e_\theta$ an isomorphism
$(\g_\ell)_{-1/2}\to(\g_\ell)_{1/2}$.  Hence
$\beta(p,q):=(e|[p,q])=([e,p]|q)$ is a nondegenerate alternating form on
$(\g_\ell)_{-1/2}$, since $(\cdot|\cdot)$ pairs
$(\g_\ell)_{1/2}$ and $(\g_\ell)_{-1/2}$ nondegenerately.  Choose
$p,q$ with $\beta(p,q)\ne0$.  Let $U_D\subset R_\ell$ be the vertex
subalgebra generated by the surviving $D_{\ell-2}$ currents; in particular
$TU_D\subset U_D$.  In $R_\ell$ all $G$-fields and $A_1$ currents vanish;
with dual bases adapted to the orthogonal sum
$\g_\ell^\natural=A_1\oplus D_{\ell-2}$, every surviving quadratic or
derivative term in \cite[(1.1)]{AKMPPStructural} lies in $U_D$.  Hence
$0=c\omega+F$ with $F\in U_D$ and
$c=-2(k_\ell+h^\vee)(e|[p,q])\in\mathbb C^\times$, since
$k_\ell+h^\vee(D_\ell)=\ell\ne0$.  Thus $\omega=-c^{-1}F\in U_D$.
The standard strong generators therefore reduce to
the $D_{\ell-2}$ currents, so $U_D=R_\ell$.  For $a,b\in D_{\ell-2}$,
\cite[Theorem~2.1(c)]{AKMPPStructural} and the normalized restriction above give
$[J^{\{a\}}{}_{\lambda}J^{\{b\}}]=J^{\{[a,b]\}}+(4-\ell)(a|b)\lambda$.
Thus the surviving currents satisfy exactly the defining affine current
relations of $V^{4-\ell}(D_{\ell-2})$ with the same normalized form; since
they strongly generate $R_\ell$, the universal property gives the surjection
\eqref{eq:universal-surj}.  This proves \textup{(ii)}.

\smallskip
\noindent\emph{Step 3: the quadratic relation.}
Take $u=e_{\varepsilon_3-\varepsilon_2}$ and
$v=e_{\varepsilon_3-\varepsilon_1}$, with roots $\alpha_u,\alpha_v$.
Since $\alpha_u(x_\theta)=\alpha_v(x_\theta)=-1/2$, both lie in
$(\g_\ell)_{-1/2}$.  Now
$\alpha_u+\alpha_v=2\varepsilon_3-\varepsilon_1-\varepsilon_2$ is neither
zero nor a $D_\ell$-root, so $[u,v]=0$.  Also
$[e,u]\in\g_{\theta+\alpha_u}$ with
$\theta+\alpha_u=\varepsilon_1+\varepsilon_3$, while
$\theta+\alpha_u+\alpha_v=2\varepsilon_3$ is neither zero nor a root;
hence $[[e,u],v]=0$.

To make the specialization of the $G$--$G$ bracket explicit, take the
constant coefficient in \cite[(1.1)]{AKMPPStructural}.  With the
normalization fixed above, it gives
\begin{align}
 G^{\{u\}}_{(0)}G^{\{v\}}
 ={}&-2(k_\ell+h^\vee)(e|[u,v])\,\omega
 +(e|[u,v])\sum_\alpha
 :J^{\{a^\alpha\}}J^{\{a_\alpha\}}:\notag\\
 &+\sum_\gamma
 :J^{\{[u,y^\gamma]^\natural\}}
  J^{\{[y_\gamma,v]^\natural\}}:
 +2(k_\ell+1)\,\partial
 J^{\{[[e,u],v]^\natural\}} .
 \label{eq:GG-zero-product}
\end{align}
Here $\{a_\alpha\}$ and $\{a^\alpha\}$ are dual bases of
$\g_\ell^\natural$ for $(\cdot|\cdot)$, while $\{y_\gamma\}$ and
$\{y^\gamma\}$ are dual bases of $(\g_\ell)_{1/2}$ for the neutral form
$\langle a,b\rangle_{\mathrm{ne}}=(f_\theta|[a,b])$, with convention
$\langle y_\gamma,y^\delta\rangle_{\mathrm{ne}}=\delta_\gamma^\delta$.
Since $[u,v]=0$, both terms multiplied by $(e|[u,v])$ vanish, including
the Virasoro term.  Since $[[e,u],v]=0$, the derivative term vanishes as
well.  Thus the only surviving constant term is the neutral contraction:
\begin{equation}\label{eq:contraction}
 G^{\{u\}}_{(0)}G^{\{v\}}=\mathcal B_\ell(u,v):=
 \sum_\gamma
 :J^{\{[u,y^\gamma]^\natural\}}
  J^{\{[y_\gamma,v]^\natural\}}:.
\end{equation}
Choose $\{y_\gamma\}$ to be the root-vector basis
$\{e_{\varepsilon_a\pm\varepsilon_j}:a=1,2,\ 3\le j\le\ell\}$.

The contraction is $\g_\ell^\natural$-equivariant: its neutral form is
invariant since $[\g_\ell^\natural,f_\theta]=0$, the projection is equivariant,
and current zero modes act derivationally on normal products.  As $u,v$ have $D_{\ell-2}$-weight
$\varepsilon_3$ and opposite $A_1$ weights, $\mathcal B_\ell(u,v)$ has
$D_{\ell-2}$-weight $2\varepsilon_3$, $A_1$ weight zero, and conformal weight $2$.
Since \eqref{eq:contraction} is already a sum of normal products of two
$\g_\ell^\natural$ currents, no $G$-field or $\omega$ occurs.  Reordering to
PBW form can only add a derivative current, whose $D_{\ell-2}$-weight is a root
or zero, never $2\varepsilon_3$.  Mixed $A_1$--$D_{\ell-2}$ and $A_1$--$A_1$
terms likewise have $D_{\ell-2}$-weight a root (or zero) and zero, respectively.
Hence only two $D_{\ell-2}$ root currents can contribute; a Cartan factor would
force the other factor to have weight $2\varepsilon_3$.
By free generation \cite[p.~2]{AKMPPStructural}, the relevant quadratic current
space has PBW basis
$$
 X_j=
 :J^{\{e_{\varepsilon_3-\varepsilon_j}\}}
  J^{\{e_{\varepsilon_3+\varepsilon_j}\}}:,
 \qquad 4\le j\le\ell.
$$
Indeed, the only pairs of roots of $D_{\ell-2}$ summing to
$2\varepsilon_3$ are
$\varepsilon_3-\varepsilon_j$ and $\varepsilon_3+\varepsilon_j$
for $j\ge4$.

For a fixed $j\ge4$, a summand in \eqref{eq:contraction}
can contribute to $X_j$ only if its two projected brackets
lie in the root spaces
$\g_{\varepsilon_3-\varepsilon_j}$ and
$\g_{\varepsilon_3+\varepsilon_j}$, in either order.
Since $u$ has root $\varepsilon_3-\varepsilon_2$,
this forces $y^\gamma$ to be proportional to
$e_{\varepsilon_2-\varepsilon_j}$ or
$e_{\varepsilon_2+\varepsilon_j}$, respectively.
With our matrix normalization,
$\langle e_{\varepsilon_1+\varepsilon_j},-e_{\varepsilon_2-\varepsilon_j}\rangle_{\rm ne}
=\langle e_{\varepsilon_1-\varepsilon_j},-e_{\varepsilon_2+\varepsilon_j}\rangle_{\rm ne}=1$,
and pairings with different $j$ vanish.  Hence precisely the following two
neutral-dual pairs can contribute to the PBW coefficient of $X_j$.
Their brackets follow from \eqref{eq:root-normalization} and
$[E_{a,b},E_{c,d}]=\delta_{bc}E_{a,d}-\delta_{da}E_{c,b}$:
$$
\begin{array}{c|c|c|c}
 y^\gamma & y_\gamma & [u,y^\gamma] & [y_\gamma,v]\\ \hline
 -e_{\varepsilon_2-\varepsilon_j} & e_{\varepsilon_1+\varepsilon_j}
 & e_{\varepsilon_3-\varepsilon_j}
 & e_{\varepsilon_3+\varepsilon_j}\\
 -e_{\varepsilon_2+\varepsilon_j} & e_{\varepsilon_1-\varepsilon_j}
 & e_{\varepsilon_3+\varepsilon_j}
 & e_{\varepsilon_3-\varepsilon_j}
\end{array}
$$
These brackets lie in $D_{\ell-2}\subset\g_\ell^\natural$, so the
projections in \eqref{eq:contraction} are trivial.  The two resulting root
currents commute and pair to zero, so their normal products agree; hence
each row contributes $X_j$, and therefore
\begin{equation}\label{eq:quadratic-image}
 \mathcal B_\ell(u,v)=2\sum_{j=4}^{\ell}X_j.
\end{equation}
Let $\eta_1,\ldots,\eta_{\ell-2}$ be the standard coordinates of
$D_{\ell-2}$ and identify $\eta_r$ with $\varepsilon_{r+2}$.  Since
$r<s$ iff $r+2<s+2$, standard positive roots map to standard positive roots;
the corresponding matrix embedding sends the normalized root vectors in
\eqref{eq:root-normalization} exactly to those on coordinates
$3,\ldots,\ell$.  As $\psi_\ell$ preserves $(-1)$-products, this gives
$\psi_\ell(v_{\ell-2})=\sum_{j=4}^{\ell}X_j$.
In $R_\ell$ both $G^{\{u\}}$ and $G^{\{v\}}$ vanish by Step~2, so
\eqref{eq:contraction} and~\eqref{eq:quadratic-image} give
$0=2\psi_\ell(v_{\ell-2})$, hence $\psi_\ell(v_{\ell-2})=0$, proving \textup{(iii)}.
\end{proof}

\begin{proposition}[Parity-free rank reduction]\label{prop:rank-reduction}
For every $\ell\ge5$, minimal reduction yields a natural surjection
$$
 Q_{\ell-2}\twoheadrightarrow
 H^0_{\mathrm{DS},f_\theta}(Q_\ell).
$$
\end{proposition}

\begin{proof}
Since $4-\ell=2-(\ell-2)$, the definition gives
$Q_{\ell-2}=V^{4-\ell}(D_{\ell-2})/\langle v_{\ell-2}\rangle$.
By Proposition~\ref{prop:local-reduction}\textup{(ii)} there is a surjection
$\psi_\ell:V^{4-\ell}(D_{\ell-2})\twoheadrightarrow R_\ell$, and part
\textup{(iii)} gives $v_{\ell-2}\in\ker\psi_\ell$.  Since
$\ker\psi_\ell$ is a vertex ideal, it contains the vertex ideal
$\langle v_{\ell-2}\rangle$ generated by this vector.  Hence the quotient
universal property gives a unique surjection
$Q_{\ell-2}\twoheadrightarrow R_\ell=H^0_{\mathrm{DS},f_\theta}(Q_\ell)$.
\end{proof}

\begin{remark}
The proof of Proposition~\ref{prop:local-reduction} holds in both
parities.  In particular, for $\ell=5$, equation~\eqref{eq:quadratic-image}
reads $\mathcal B_5(u,v)=2(X_4+X_5)$ and gives the endpoint
$D_5\to D_3$ of the odd-rank chain.
\end{remark}

\section{Odd-rank induction and maximal ideals}
\label{sec:odd}

We now assume that $\ell$ is odd.

\subsection{The base case}
\label{sec:base}

The odd chain terminates at $D_3\cong A_3=\mathfrak{sl}_4$ and $k=-1$.
To make the identification of the quadratic vector explicit, let
$\eta_1,\eta_2,\eta_3$ be the orthonormal coordinates for this $D_3$,
and write its simple roots as $\beta_1=\eta_1-\eta_2$,
$\beta_2=\eta_2-\eta_3$ and $\beta_3=\eta_2+\eta_3$.
Let $\alpha_1,\alpha_2,\alpha_3$ be the consecutive simple roots of $A_3$.
Choose the accidental isomorphism $D_3\cong A_3$ by
$\beta_1\mapsto\alpha_2$, $\beta_2\mapsto\alpha_1$ and
$\beta_3\mapsto\alpha_3$.  The pullback of the normalized $A_3$ form is
invariant on $D_3$ and has roots of squared length $2$, hence equals our
normalized form by uniqueness.  Thus this isomorphism preserves the affine
cocycle and extends at level $-1$ to
$V^{-1}(D_3)\cong V^{-1}(\mathfrak{sl}_4)$.  Moreover,
$\varpi^{D_3}_1=\eta_1$ maps to $\varpi^{A_3}_2$, and hence the finite
highest weight $2\varpi^{D_3}_1$ of $v_3$ maps to $2\varpi^{A_3}_2$.
Furthermore,
$$
 v_3=
 e_{\eta_1-\eta_2}(-1)e_{\eta_1+\eta_2}(-1)\mathbf1
 +e_{\eta_1-\eta_3}(-1)e_{\eta_1+\eta_3}(-1)\mathbf1.
$$
To fix the signs, write $F_{a,b}$ for the standard $4\times4$ matrix
units, and choose the Lie algebra isomorphism sending
$e_{\beta_1},e_{\beta_2},e_{\beta_3}$ to $F_{2,3},F_{1,2},F_{3,4}$,
respectively.  Taking brackets, with the normalization
\eqref{eq:root-normalization}, gives
$$
 \begin{aligned}
 e_{\eta_1-\eta_2}&\longmapsto F_{2,3},&
 e_{\eta_1+\eta_2}&\longmapsto F_{1,4},\\
 e_{\eta_1-\eta_3}&\longmapsto-F_{1,3},&
 e_{\eta_1+\eta_3}&\longmapsto F_{2,4}.
 \end{aligned}
$$
For example, $e_{\eta_1-\eta_3}=[e_{\beta_1},e_{\beta_2}]$ maps to
$[F_{2,3},F_{1,2}]=-F_{1,3}$.
Thus, with the standard positive root vectors in $\mathfrak{sl}_4$,
the image of $v_3$ is exactly
$$
 e_{\alpha_1+\alpha_2+\alpha_3}(-1)e_{\alpha_2}(-1)\mathbf1
 -e_{\alpha_2+\alpha_3}(-1)e_{\alpha_1+\alpha_2}(-1)\mathbf1.
$$
For $l=3$, $n=1$ this is $v_{3,1}$ of \cite[(5.1)]{AdamovicPerse},
singular by \cite[Theorem~4.1]{AdamovicPerse} and of finite weight
$2\varpi_2$.  For $n=4$, the element $v_1$ of
\cite[\S7]{ArakawaMoreau} is
$e_{\theta}e_{\alpha_2}-e_{\alpha_1+\alpha_2}e_{\alpha_2+\alpha_3}$;
the two factors in each product commute, so its symmetrization $\sigma(v_1)$
is exactly the displayed state.  Theorem~7.2 of loc. cit. says that
$\sigma(v_1)$ generates the maximal affine submodule of
$V^{-1}(\mathfrak{sl}_4)$.  Since $-1+h^\vee(A_3)=3\ne0$, the Sugawara
current-mode argument of Lemma~\ref{lem:proper} applies verbatim, so
$U(\widehat{\mathfrak{sl}}_4)\sigma(v_1)=\langle\sigma(v_1)\rangle$ as a
vertex ideal; hence the quotient is $L_{-1}(\mathfrak{sl}_4)$.
The Hamiltonian-reduction argument for the same base case is also reviewed in
\cite[\S5.1]{APV}.  Consequently
\begin{equation}\label{eq:base}
 Q_3\cong L_{-1}(D_3).
\end{equation}

\subsection{The induction}

\begin{theorem}\label{thm:odd}
For every odd integer $\ell\ge3$, one has
$Q_\ell\cong L_{2-\ell}(D_\ell)$.  In particular, for every odd
$\ell\ge5$, $\Rad V^{2-\ell}(D_\ell)=\langle v_\ell\rangle$.
\end{theorem}

\begin{proof}
The case $\ell=3$ is \eqref{eq:base}.  Let now $\ell\ge5$ be odd and assume
inductively that $Q_{\ell-2}\cong L_{4-\ell}(D_{\ell-2})$.  By
Proposition~\ref{prop:rank-reduction}, there is a surjection
$L_{4-\ell}(D_{\ell-2})\cong Q_{\ell-2}\twoheadrightarrow R_\ell$.
Lemma~\ref{lem:proper} makes $Q_\ell$ nonzero; as a graded quotient of the
vacuum module it is lower-bounded with finite-dimensional homogeneous
$D_\ell$-components, while $k_\ell=2-\ell<0$ and
$k_\ell\ne-h^\vee(D_\ell)$.  Thus Lemma~\ref{lem:detection} gives
$R_\ell=H_{\mathrm{DS}}(Q_\ell)\ne0$.  The kernel of the displayed
surjection is therefore a proper ideal of the simple source, hence zero.
Consequently $R_\ell\cong Q_{\ell-2}\cong L_{4-\ell}(D_{\ell-2})$, so
$R_\ell=H^0_{\mathrm{DS},f_\theta}(Q_\ell)$ is simple.

It remains only to verify explicitly the hypotheses of
Proposition~\ref{prop:maximality}.  By Lemma~\ref{lem:proper},
$\langle v_\ell\rangle$ is a proper graded ideal of
$V^{2-\ell}(D_\ell)$.  Since $L_{2-\ell}(D_\ell)$ is simple, $\Rad V^{2-\ell}(D_\ell)$ is
maximal among proper graded ideals.  If $J$ is proper graded, then
$J+\Rad V^{2-\ell}(D_\ell)$ is still proper because its degree-zero part is
zero; hence $J\subset\Rad V^{2-\ell}(D_\ell)$.  Taking
$J=\langle v_\ell\rangle$ gives the required inclusion.
Thus all hypotheses of Proposition~\ref{prop:maximality} are satisfied, and
it yields both $Q_\ell\cong L_{2-\ell}(D_\ell)$ and
$\langle v_\ell\rangle=\Rad V^{2-\ell}(D_\ell)$.
\end{proof}

\begin{corollary}[Odd DS chain]\label{cor:odd-chain}
For every odd $\ell\ge5$, the natural surjection of
Proposition~\ref{prop:rank-reduction} is an isomorphism
$Q_{\ell-2}\xrightarrow{\sim}H^0_{\mathrm{DS},f_\theta}(Q_\ell)$:
Theorem~\ref{thm:odd} makes the source simple, while
Lemmas~\ref{lem:proper} and~\ref{lem:detection} make the target nonzero.
At the step $D_m\to D_{m-2}$ the lower-rank level is
$4-m=2-(m-2)$.  Thus successive minimal reductions satisfy
$H^0_{\mathrm{DS},f_{\theta_m}}(Q_m)\cong Q_{m-2}$ for
$m=\ell,\ell-2,\ldots,5$, and terminate at $Q_3$, where $f_{\theta_m}$
denotes the chosen minimal nilpotent of $D_m$.
\end{corollary}

\begin{remark}[Compatibility with the known collapse]
For odd $\ell\ge5$, Theorem~\ref{thm:odd} and
Corollary~\ref{cor:odd-chain} give
$H^0_{\mathrm{DS},f_\theta}(L_{2-\ell}(D_\ell))\cong
L_{4-\ell}(D_{\ell-2})$.  This agrees exactly with
\cite[Theorem~2.3, Table~4, (2.6), and (2.8)]{AKMPP}: for
$D_\ell=\mathfrak{so}(2\ell)$ one has
$p(k)=(k+2)(k+\ell-2)$, while the $A_1\oplus D_{\ell-2}$ component
levels at $k=2-\ell$ are $0$ and $4-\ell$; (2.8) discards the zero-level
factor.  Here the collapse is recovered from quotient-level rank reduction
rather than used in the proof.
\end{remark}

\begin{remark}[Relation with the even-rank argument]
Section~\ref{sec:rank} isolates the quadratic-current step of
\cite[\S8.3]{JinAM}: with the same $u,v$, neutral pairing and matrix
normalization, the $G$--$G$ contraction is twice the lower-rank quadratic
vector.  This step is parity-free.  Even rank additionally carries the two
Pfaffian relations through reduction \cite[\S8.2]{JinAM} and ends at $D_4$;
odd rank has no Pfaffian relations and ends at the simple $D_3\cong A_3$
quadratic quotient.
\end{remark}

\bigskip
\noindent\textsc{Department of Mathematics, Sichuan University}\\
\textit{Email address:} \texttt{jinsihai@stu.scu.edu.cn}

\end{document}